\documentclass[runningtargets,a4paper]{llncs}

\usepackage{amsmath,amsfonts,amssymb}
\usepackage[version=3]{mhchem}
\usepackage[all]{xy} 
\usepackage{marginnote}
\usepackage{hyperref}
\newcommand{\keywords}[1]{\par\addvspace\baselineskip\noindent\keywordname\enspace\ignorespaces#1}

\DeclareMathOperator{\spann}{span}

\DeclareMathOperator{\im}{im} 
\DeclareMathOperator{\sign}{sign}
\DeclareMathOperator{\e}{e} 
\DeclareMathOperator{\diag}{diag}

\newcommand{\RR}{{\mathbb R}}
\newcommand{\QQ}{{\mathbb Q}}

\newcommand{\mal}{\circ} 

\newcommand{\dd}[2]{\frac{\text{d} #1}{\text{d} #2}}

\newcommand{\mat}{W}
\newcommand{\vct}{w}

\newcommand{\Ys}{Y} 
\newcommand{\Yk}{\tilde{Y}} 
\newcommand{\ys}{y} 
\newcommand{\yk}{\tilde{y}} 

\newcommand{\Vs}{V_s} 

\newcommand{\Mi}{I_E} 
\newcommand{\Ms}{I_s} 
\newcommand{\Mr}{\Delta_k} 
\newcommand{\ML}{A_k} 

\newcommand{\kk}{\kappa_k} 
\newcommand{\ka}{\kappa} 

\begin{document}

\mainmatter 

\title{Generalized Mass-Action Systems \\
and Positive Solutions of Polynomial Equations \\
with Real and Symbolic Exponents}

\titlerunning{Polynomial Equations with Real and Symbolic Exponents}

%
%
\author{Stefan M{\"u}ller \and Georg Regensburger}
\authorrunning{S.~M{\"u}ller \and G.~Regensburger}

\institute{Johann Radon Institute for Computational and Applied Mathematics (RICAM),\\
Austrian Academy of Sciences, Linz, Austria\\
\email{\{stefan.mueller,georg.regensburger\}@ricam.oeaw.ac.at}}

\maketitle

\begin{abstract}
Dynamical systems arising from chemical reaction networks with mass action kinetics
are the subject of chemical reaction network theory (CRNT). In particular, this theory provides statements
about uniqueness, existence, and stability of positive steady states
for all rate constants and initial conditions.
In terms of the corresponding polynomial equations,
the results guarantee uniqueness and existence of positive solutions for all positive parameters.

We address a recent extension of CRNT, called generalized mass-action systems,
where reaction rates are allowed to be power-laws in the concentrations.
In particular, the (real) kinetic orders can differ from the (integer) stoichiometric coefficients.
As with mass-action kinetics, complex balancing equilibria are determined by the graph Laplacian of the underlying network
and can be characterized by binomial equations and parametrized by monomials.
In algebraic terms, we focus on a constructive characterization of positive solutions
of polynomial equations with real and symbolic exponents.

Uniqueness and existence for all rate constants and initial conditions
additionally depend on sign vectors of the stoichiometric and kinetic-order subspaces.
This leads to a generalization of Birch's theorem,
which is robust with respect to certain perturbations in the exponents.
In this context, we discuss the occurrence of multiple complex balancing equilibria.

We illustrate our results by a running example
and provide a MAPLE worksheet with implementations of all algorithmic methods.

\end{abstract}

\section{Introduction}
In this work, we focus on dynamical systems arising from (bio-)chemical reaction networks with \emph{generalized} mass-action kinetics
and positive solutions of the corresponding systems of generalized polynomial equations. 

In chemical reaction network theory, as initiated by Horn, Jackson, and Feinberg in the 1970s~\cite{Feinberg1972,HornJackson1972,Horn1972},
several fundamental results are based on the assumption of mass action kinetics (MAK). 
Consider the reaction
\begin{equation} \label{ABC}
1 \, \ce{A} + 1 \, \ce{B} \to \ce{C}
\end{equation}
involving the reactant species A, B and the product C,
where we explicitly state the stoichiometric coefficients of the reactants.
The left- and right-hand sides of a reaction, in this case A+B and C, are called (stoichiometric) complexes.
Let 
\[
[A]=[A](t)
\]
denote the concentration of species $\ce{A}$ at time $t$, and analogously for $\ce{B}$ and $\ce{C}$.
Assuming MAK, the rate at which the reaction occurs is given by
\[
v = k \, [\ce{A}]^1 [\ce{B}]^1
\]
with rate constant $k>0$. In other words,
the reaction rate is a monomial in the reactant concentrations $[\ce{A}]$ and $[\ce{B}]$
with the stoichiometric coefficients as exponents.
Within a network involving additional species and reactions,
the above reaction contributes to the dynamics of the species concentrations as
\[
\dd{}{t}
\begin{pmatrix}
[\ce{A}] \\ [\ce{B}] \\ [\ce{C}] \\ [\ce{D}] \\ \vdots
\end{pmatrix}
=
k \, [\ce{A}] [\ce{B}] \!
\begin{pmatrix}
-1 \\ -1 \\ 1\\ 0 \\ \vdots
\end{pmatrix}
+ \cdots
\]

In many applications, the reaction network is given,
but the values of the rate constants are unknown.
Surprisingly, there are results on existence, uniqueness, and stability of steady states
that do not depend on the rate constants.
See, for example, the lecture notes \cite{Feinberg1979} and the surveys \cite{Feinberg1987,Feinberg1995a,Gunawardena}. 

However, the validity of MAK is limited;
it only holds for elementary reactions in homogeneous and dilute solutions.
For biochemical reaction networks in intracellular environments, the rate law has to be modified.
In previous work  \cite{MuellerRegensburger2012}, we allowed generalized mass-action kinetics (GMAK)
where reaction rates are power-laws in the concentrations.
In particular, the exponents need not coincide with the stoichiometric coefficients
and need not be integers.
For example, the rate at which reaction \eqref{ABC} occurs may be given by
\[
v = k \, [\ce{A}]^a [\ce{B}]^b
\]
with kinetic orders $a,b > 0$.
Formally, we specify the rate of a reaction by associating (here indicated by dots) with the reactant complex a kinetic complex,
which determines the exponents in the generalized monomial:
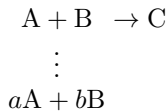
\[
\begin{array}{ccc} 
\ce{A} + \ce{B} & \to & \ce{C}  \\
\vdots \\
a \ce{A} + b \ce{B}
\end{array}
\]
Before we give the definition of generalized mass action systems, we introduce a running example, 
which will be used to motivate and illustrate general statements.
Throughout the paper, we focus on algorithmic aspects of the theoretical results.
Additionally, we provide a MAPLE worksheet\footnote{The worksheet is available at \url{http://gregensburger.com/software/GMAK.zip}.} with implementations of all algorithms applied to the running example. 
For other applications of computer algebra to chemical reaction networks,
we refer to \cite{BoulierLemairePetitotSedoglavic2011,ErramiSeilerEiswirthWeber2012,Lemaire2012,SamalErramiWeber2012}.

\bigskip

\noindent {\bf Notation.} We denote the strictly positive real numbers by $\RR_>$.
We define $\e^x \in \RR^n_>$ for $x \in \RR^n$ component-wise, that is, $(\e^x)_i = \e^{x_i}$;
analogously, $\ln(x) \in \RR^n$ for $x \in \RR^n_>$ and $x^{-1} \in \RR^n$ for $x \in \RR^n$ with $x_i \neq 0$.
For $x,y \in \RR^n$,
we denote the component-wise (or Hadamard) product by $x \mal y \in \RR^n$, that is, $(x \mal y)_i = x_i y_i$;
for $x \in \RR^n_>$ and $y \in \RR^n$, we define $x^y \in \RR_>$ as $\prod_{i=1}^n x_i^{y_i}$.

Given a matrix $B \in \RR^{n \times m}$, we denote by $b^1, \dots, b^m$ its column vectors and by $b_1, \dots, b_n$ its row vectors.
For $x \in \RR^n_>$, we define $x^B \in \RR^m_>$ as
\[
(x^B)_j  = x^{b^j} = \prod_{i=1}^n x_i^{b_{ij}}
\]
for $j=1,\ldots,m$. As a consequence,
\[
\ln (x^B) = B^T \ln x .
\]
Finally, we identify a matrix $B \in \RR^{n \times m}$ with the corresponding linear map $B \colon \RR^m \to \RR^n$
and write $\im(B)$ and $\ker(B)$ for the respective vector subspaces. 

\section{Running Example}

We consider a reaction network based on the weighted directed graph
\begin{equation} \label{graphex}
\xymatrix{
1 \ar@<0.5ex>[r]^{k_{12}} & 2 \ar@<0.5ex>[l]^{k_{21}} \ar[d]^{k_{23}} & 4 \ar@<0.5ex>[r]^{k_{45}} & 5 \ar@<0.5ex>[l]^{k_{54}} \\
& 3 \ar[lu]^{k_{31}} 
}
\end{equation}
with 5 vertices, 6 edges and corresponding positive weights.
Clearly, the edges represent reactions and the weights are rate constants.
We assume that the network contains 4 species A, B, C, D
and associate with each vertex a (stoichiometric) complex,
that is, a formal sum of species:
\begin{equation*}
\xymatrix{
\ce{A}+\ce{B} \ar@<0.5ex>[r] & \ce{C} \ar@<0.5ex>[l] \ar[d] & \ce{A} \ar@<0.5ex>[r] & \ce{D} \ar@<0.5ex>[l] \\
& 2 \ce{A} \ar[lu] 
}
\end{equation*}
In order to specify the reaction rates, e.g., $v_{12}=k_{12} [\ce{A}]^\frac{1}{2} [\ce{B}]^\frac{3}{2}$,
we additionally associate a kinetic complex with each source vertex:
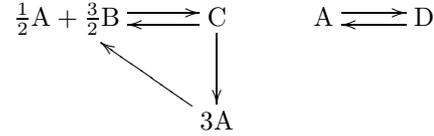
\begin{equation*}
\xymatrix{
\frac{1}{2} \ce{A} + \frac{3}{2} \ce{B} \ar@<0.5ex>[r] & \ce{C} \ar@<0.5ex>[l] \ar[d] & \ce{A} \ar@<0.5ex>[r] & \ce{D} \ar@<0.5ex>[l] \\
& 3 \ce{A} \ar[lu] 
}
\end{equation*}
Writing 
\[
x=(x_1,x_2,x_3,x_4)^T
\]
for the concentrations of species A, B, C, D, the dynamics of the generalized mass action system is given by
\begin{equation}
\label{eq:ODE}
\dd{}{t}
\begin{pmatrix}
x_1 \\ x_2 \\ x_3 \\ x_4
\end{pmatrix}
=
\begin{pmatrix}
-1 & 1  &  2 & -1 & -1 &  1\\
-1 & 1  &  0 &  1 &  0 &  0\\
 1 & -1 & -1 &  0 &  0 &  0\\
 0 &  0 &  0 &  0 &  1 &  -1
\end{pmatrix}
\begin{pmatrix}
k_{12} \, (x_1)^\frac{1}{2} (x_2)^\frac{3}{2} \\ k_{21} \, x_3 \\ k_{23} \, x_3 \\ k_{31} \, (x_1)^3 \\ k_{45} \, x_1 \\ k_{54} \, x_4
\end{pmatrix}
= N \, v(x) ,
\end{equation}
where we fix an order on the edges, $E = \big( (1,2),(2,1),(2,3),(3,1),(4,5),(5,4) \big)$,
and introduce the stoichiometric matrix $N$ and the vector of reaction rates $v(x)$.

We further decompose the system.
Writing the stoichiometric and kinetic complexes as column vectors of the matrices
\[
\Ys =
\begin{pmatrix}
1 & 0 & 2 & 1 & 0\\
1 & 0 & 0 & 0 & 0\\
0 & 1 & 0 & 0 & 0\\ 
0 & 0 & 0 & 0 & 1
\end{pmatrix}
\quad\text{and}\quad 
\Yk =
\begin{pmatrix}
\frac{1}{2} & 0 & 3 & 1 & 0\\
\frac{3}{2} & 0 & 0 & 0 & 0\\ 
0           & 1 & 0 & 0 & 0\\ 
0           & 0 & 0 & 0 &1
\end{pmatrix} 
\] 
and using the incidence matrix of the graph $\eqref{graphex}$,
\[
\Mi= 
\begin{pmatrix}
-1 &  1 &  0 &  1 &  0 &  0\\
 1 & -1 & -1 &  0 &  0 &  0\\ 
 0 &  0 &  1 & -1 &  0 &  0\\ 
 0 &  0 &  0 &  0 & -1 &  1\\ 
 0 &  0 &  0 &  0 &  1 & -1
\end{pmatrix} ,
\] 
we can write the stoichiometric matrix as
\[
N = \Ys \, \Mi .
\]
The vector of reaction rates $v(x)$ can also be decomposed by introducing 
a diagonal matrix 
\[
\Mr=\diag(k_{12},k_{21},k_{23},k_{31},k_{45},k_{54})
\] 
containing the rate constants,
a matrix indicating the source vertex of each reaction, 
\[
\Ms=
\begin{pmatrix}
1 & 0 & 0 & 0 & 0 & 0\\ 
0 & 1 & 1 & 0 & 0 & 0\\ 
0 & 0 & 0 & 1 & 0 & 0\\
0 & 0 & 0 & 0 & 1 & 0\\
0 & 0 & 0 & 0 & 0 & 1
\end{pmatrix} ,
\]
and the vector of monomials determined by the kinetic complexes,
\[
x^{\Yk} =
\begin{pmatrix}
(x_1)^\frac{1}{2} (x_2)^\frac{3}{2} \\ x_3 \\ (x_1)^3 \\ x_1 \\ x_4
\end{pmatrix} .
\]
Then,
\[
v(x) = \Mr \, \Ms^T x^{\Yk} ,
\]
and we can write
\[
\dd{x}{t}
= N \, v(x) = \Ys \,\Mi \, \Mr \,\Ms^T \, x^{\Yk} . 
\]
Note that the matrix
\begin{equation}
\label{eq:Ak}
\ML=\Mi \, \Mr\, \Ms^T =
\begin{pmatrix}
-k_{12} &           k_{21} &  k_{31} &       0 &       0 \\ 
 k_{12} & -(k_{21}+k_{23}) &       0 &       0 &       0 \\ 
      0 &           k_{23} & -k_{31} &       0 &       0 \\ 
      0 &                0 &       0 & -k_{45} &  k_{54} \\ 
      0 &                0 &       0 &  k_{45} & -k_{54}
\end{pmatrix}
\end{equation}
depends only on the weighted digraph,
while $\Ys$ and $x^{\Yk}$ are determined by the stoichiometric and kinetic complexes.
The resulting decomposition
\begin{equation*}
\dd{x}{t} = \Ys \ML \, x^{\Yk} 
\end{equation*}
is due to~\cite{HornJackson1972}, where $\ML$ is called kinetic matrix and the stoichiometric and kinetic complexes are equal, that is, $\Ys=\Yk$. 
The interpretation of $\ML$ as a weighted graph Laplacian was introduced in~\cite{GatermannWolfrum2005}
and used in~\cite{CraciunDickensteinShiuSturmfels2009,ThomsonGunawardena2009,Gunawardena2012,MirzaevGunawardena2013,Johnston2013},
in particular, in connection with the matrix-tree theorem.

\section{Generalized Mass Action Systems}
\label{sec:generalized}

We consider \emph{directed graphs} $G=(V,E)$ given by a finite set of \emph{vertices} 
\[
V = \{ 1,\ldots,m \}
\] 
and a finite set of \emph{edges} $E\subseteq V \times V$.
We often denote an edge $e=(i,j) \in E$ by $i \to j$ to emphasize that it is directed from the \emph{source} $i$ to the \emph{target} $j$.
Further, we write
\[
\Vs = \{i \mid i\to j \in E \}
\]
for the set of source vertices that appear as a source of some edge.

\begin{definition}
A \emph{generalized chemical reaction network} $(G, \ys, \yk)$ is given by a digraph $G=(V,E)$ without self-loops,
and two functions
\[
\ys \colon V \rightarrow \RR^n \quad \text{and} \quad \yk \colon \Vs \rightarrow \RR^n
\]
assigning to each vertex a \emph{(stoichiometric) complex} and to each source a \emph{kinetic complex}.
\end{definition}
We note that this definition differs from \cite{MuellerRegensburger2012}. On the one hand, kinetic complexes were assigned also to non-source vertices, on the other hand, all (stoichiometric) complexes had to be different, and analogously the kinetic complexes.  

\begin{definition}
A \emph{generalized mass action system} $(G_k, \ys, \yk)$ is a generalized chemical reaction network $(G, \ys, \yk)$,
where edges $(i,j)\in E$ are labeled with {\em rate constants} $k_{ij} \in \RR_>$.
\end{definition}

The contribution of reaction $i \to j \in E$
to the dynamics of the species concentrations $x \in \RR^n$ is proportional to the \emph{reaction vector} $\ys(j) - \ys(i) \in \RR^n$.
Assuming generalized mass action kinetics,
the rate of the reaction is determined by the source kinetic complex $\yk(i)$ and the positive rate constant $k_{ij}$:
\[
v_{i \to j}(x) = k_{ij} \, x^{\yk(i)} .
\]

The ordinary differential equation associated with a generalized mass action system is defined as 
\begin{equation*}
\dd{x}{t}
= \sum_{i \to j  \in E} k_{ij} \, x^{\yk(i)} \big( \ys(j) - \ys(i) \big) .
\end{equation*}
The change over time lies in the {\em stoichiometric subspace}
\begin{equation*}
S = \spann \{ \ys(j) - \ys(i) \in \RR^n \mid i \to j \in E \} ,
\end{equation*}
which suggests the definition of a (positive) \emph{stoichiometric compatibility class} $(c'+S) \cap \RR^n_>$ with $c' \in \RR^n_>$.

In case every vertex is a source, that is, $\Vs = V$,
we introduce also the {\em kinetic-order subspace}
\begin{equation*}
\tilde{S} = \spann \{ \yk(j) - \yk(i) \in \RR^n \mid i \to j \in E \} .
\end{equation*}

In order to decompose the right-hand side of the ODE system,
we define the matrices $\Ys \in \RR^{n \times m}$ as $\ys^j = \ys(j)$
and $\Yk \in \RR^{n \times m}$ as $\yk^j = \yk(j)$ for $j \in \Vs$ and $\yk^j = 0$ otherwise (see also the remark below).
Further, we introduce the weighted \emph{graph Laplacian} $\ML \in \RR^{m \times m}$:
$(\ML)_{ij} = k_{ji}$ if $j \to i \in E$, $(\ML)_{ii} = - \sum_{i \to j \in E} k_{ij}$, and $(\ML)_{ij} = 0$ otherwise.
We obtain:
\begin{equation*}
\dd{x}{t}
=
\Ys \ML \, x^{\Yk} .
\end{equation*}
Note that $\yk^j$ can be chosen arbitrarily for $j \notin \Vs$, since in this case $(\ML)^j = 0$ and hence $(\ML)^j x^{\yk^j} = 0$.

Steady states of the ODE satisfying $x \in \RR^n_>$ and $\ML \, x^{\Yk} = 0$
are called {\em complex balancing equilibria}.
We denote the corresponding set by
\begin{equation*}
Z_k = \{ x \in \RR^n_> \mid \ML \, x^{\Yk} = 0 \} .
\end{equation*}

Finally, the {\em (stoichiometric) deficiency} is defined as
\[
\delta = m-l - s ,
\]
where $m$ is the number of vertices, $l$ is the number of connected components, 
and $s=\dim S$ is the dimension of the stoichiometric subspace.

Using $S = \im(Y \, \Mi)$, where $\Mi$ is the incidence matrix of the graph (for a fixed order on $E$),
we obtain the equivalent definition 
$$\delta = \dim(\ker(Y) \cap \im(\Mi)),$$
see for example \cite{Johnston2013}.
Further, note that $\im(\ML) \subseteq \im(\Mi)$.
Now, if $\delta=0$, then $\ker(Y) \cap \im(\ML) \subseteq \ker(Y) \cap \im(\Mi) = \{0\}$,
and there are no $x \in \RR^n_>$ such that $\Ys \ML \, x^{\Yk} = 0$, but $\ML \, x^{\Yk} \neq 0$.
In other words, if $\delta=0$, there are no steady states other than complex balancing equilibria.

\section{Graph Laplacian}
\label{sec:lap}

A basis for the kernel of $\ML$ in $\eqref{eq:Ak}$ is given by
\[
(k_{31}\,k_{21}+k_{31}\,k_{23}, k_{12}\,k_{31}, k_{23}\,k_{12}, 0, 0)^T
\quad \text{and} \quad
(0,0,0,k_{54},k_{45})^T .
\]
Obviously, the support of the vectors coincides with the connected components of the graph.
In general, this holds for the strongly connected components without outgoing edges.

Let $G_k=(V,E,k)$ be a weighted digraph without self-loops and $\ML$ its graph Laplacian.
Further, let $l$ be the number of connected components (aka linkage classes)
and $T_1, \ldots, T_t \subseteq V$ be the sets of vertices
within the strongly connected components without outgoing edges (aka terminal strong linkage classes).
Clearly, $t \ge l$.
A fundamental result of CRNT~\cite{FeinbergHorn1977} states
that there exist linearly independent $\chi^1,\ldots,\chi^t \in \RR^n_\ge$,
where $\chi^\lambda_\mu > 0$ if $\mu \in T_\lambda$ and $\chi^\lambda_\mu = 0$ otherwise,
such that $\ker(\ML) = \spann\{ \chi^1,\ldots,\chi^t \}$.

In fact, the non-zero entries in the basis vectors can be computed using the matrix-tree theorem:
\[
\chi^\lambda_\mu = K_\mu ,
\quad \lambda \in \{1, \ldots, t \}
\]
with \emph{tree constants}
\[
K_\mu = \sum_{\mathcal{T} \in \mathcal{S}_\mu} \prod_{i \to j \in \mathcal{T}} k_{ij} ,
\quad \mu \in \{1, \ldots, m \} ,
\]
where $\mathcal{S}_\mu$ is the set of directed spanning trees (for the respective strongly connected component without outgoing edges) rooted at vertex $\mu$;
see~\cite{Gunawardena2012,MirzaevGunawardena2013,Johnston2013}.
We refer to \cite{BrualdiRyser1991} for further details and references on the graph Laplacian
and a combinatorial proof of the matrix-tree theorem following \cite{Zeilberger1985}.

If there exists $\psi \in \RR^m_>$ with $\ML \, \psi = 0$,
then every vertex resides in a strongly connected component without outgoing edges,
that is, every connected component is strongly connected. 
In this case, the underlying unweighted digraph is called {\em weakly reversible}.
Now, let $(G,\ys,\yk)$ be a generalized chemical reaction network.
If there exist rate constants $k$ such that the generalized mass action system $(G_k,\ys,\yk)$ admits a complex balancing equilibrium~$x \in \RR^n_>$,
that is, $\ML \, x^{\Yk} = 0$,
then $G$ is weakly reversible.

\section{Binomial Equations for Complex Balancing Equilibria}

For a weakly reversible digraph,
we know from the previous section that a basis for $\ker(\ML)$, parametrized by the weights,
is given in terms of the $l$~connected components and the $m$~tree constants.

In our example, where $l=2$ and $m=5$,
basis vectors of $\ker(\ML)$ are given by
\[
(K_1,K_2,K_3, 0, 0)^T
\quad \text{and} \quad
(0,0,0,K_4,K_5)^T
\]
with tree constants
\[
(K_1,K_2,K_3,K_4,K_5) = (k_{31}\,k_{21}+k_{31}\, k_{23}, k_{12}\,k_{31}, k_{23} \, k_{12}, k_{54},k_{45}) .
\]

Due to their special structure, we immediately find ``binomial'' basis vectors for the orthogonal complement $\ker(\ML)^\perp$,
\[
(-K_2,K_1,0, 0, 0)^T ,
\quad
(0,-K_3,K_2, 0, 0)^T ,
\quad \text{and} \quad
(0,0,0,-K_5,K_4)^T ,
\]  
which are again determined by the connected components and tree constants.
These vectors form a basis since they are linearly independent and
\[
\dim \ker(\ML)^\perp = m - \dim \ker(\ML) = m-l = 5-2 = 3 .
\]
In our example, a complex balancing equilibrium $x \in \RR^4_>$ with $\psi = x^{\Yk}$ and hence $\ML \, \psi = 0$,
can equivalently be described as a positive solution of the binomial equations
\[
\begin{pmatrix}
-K_2 & K_1 & 0 & 0 & 0 \\
0 & -K_3 & K_2 & 0 & 0 \\
0 & 0 & 0 & -K_5 & K_4
\end{pmatrix}
\psi = 0 .
\]
In other words, $\psi \in \ker(\ML)$ is equivalent to $\psi \perp \ker(\ML)^\perp$ or a basis thereof.
Explicitly, we have $\psi = x^{\Yk} = ( (x_1)^\frac{1}{2} (x_2)^\frac{3}{2},\, x_3,\, (x_1)^3,\, x_1,\, x_4 )^T$
and
\begin{equation}
\label{eq:bineq}
K_1 \, x_{{3}}-K_2 \, (x_1)^\frac{1}{2}(x_2)^\frac{3}{2} = 0, \quad K_2 \, (x_1)^{3}-K_3 \, x_3  = 0, \quad K_4 \, x_4-K_5 \, x_1  = 0 .
\end{equation}

Clearly, these considerations generalize to arbitrary weakly reversible digraphs:
Based on the (strongly) connected components,
we can characterize complex balancing equilibria by $m-l$ binomial equations with tree constants as coefficients.

\begin{proposition} \label{prop:binomeq}
Let $\ML$ be the graph Laplacian of a weakly reversible digraph with positive weights
and $m$ vertices ordered within $l$ connected components,
\[
L_\lambda = (i^\lambda_\mu)_{\mu=1,\ldots,m_\lambda}
\quad \text{for } \lambda = 1,\ldots,l ,
\quad \text{where } \textstyle \sum_{\lambda = 1}^l m_\lambda = m .
\]
Let $\Yk \in \RR^{n \times m}$ and    
\[
Z_k = \{ x \in \RR^n_> \mid \ML \, x^{\Yk}=0 \}.
\]
Then, 
\begin{equation*}
Z_k = \{ x \in \RR^n_> \mid K_i \, x^{\yk^j}-K_j \, x^{\yk^i} =0 , \; (i,j) \in \mathcal{E} \}
\end{equation*}  
where
\[
\mathcal{E} = \{ (i^\lambda_\mu,i^\lambda_{\mu+1}) \mid \lambda=1,\ldots,l; \; \mu=1,\ldots,m_\lambda-1 \} .
\]
\end{proposition}

Note that the actual binomial equations depend on the order of the vertices within the connected components,
but the zero set does not.

\section{Binomial Equations with Real and Symbolic Exponents}
\label{sec:possol}

In this section, we collect basic facts about positive real solutions of binomial equations with real exponents.
We present the results in full generality, in particular, not restricted to complex balancing equilibria,
and emphasize algorithmic aspects.
Moreover, by reducing computations to linear algebra,
we outline the treatment of symbolic exponents. 

In an algebraic perspective, one usually considers solutions of binomial equations with integer exponents.
We refer to \cite{Dickenstein2009} for an introduction including algorithmic aspects and an extensive list of references. 
An algorithm with polynomial complexity for computing solutions with non-zero or positive coordinates of parametric binomial systems is presented in \cite{GrigorievWeber2012}.
For recent algorithmic methods for binomial equations and monomial parametrizations, see \cite{AdrovicVerschelde2013}. Toric geometry and computer algebra was introduced to the study of mass action systems in \cite{Gatermann2001,GatermannHuber2002,GatermannEiswirthSensse2005}
and further developed in \cite{CraciunDickensteinShiuSturmfels2009}.
So-called \emph{toric steady states} are solutions of binomial equations arising from polynomial dynamical systems \cite{PerezMillanDickensteinShiuConradi2012}.

In chemical reaction networks, it is natural to consider real exponents:
kinetic orders, measured by experiments, need not be integers.
Also in S-systems~\cite{Savageau1969b,Voit2013},
defined by binomial power-laws, the exponents are real numbers identified from data.
We note that binomial equations are implicit in the original works on chemical reaction networks \cite{HornJackson1972,Horn1972}. 

In the following, we consider binomial equations  
\begin{align*}
\alpha_i \, x^{a^i} - \beta_i \, x^{b^i} = 0 \quad \textrm{ for } i = 1, \ldots, r 
\end{align*}
for $x \in \RR^n_>$, where $a^i, b^i \in \RR^n$ and $\alpha_i, \beta_i \in \RR_>$. Clearly, $x$ is a solution 
iff
\[
x^{a^i-b^i} = \frac{\beta_i}{\alpha_i} \quad \textrm{for } i=1,\ldots,r .
\]
By introducing the exponent matrix $M \in \RR^{n \times r}$, whose $i$th column is the vector $a^i-b^i$, and the vectors $\alpha,\beta \in \RR_>^r$ with entries $\alpha_i$ and $\beta_i$, respectively, 
we can rewrite the above equation system as
\[
x^M = \frac{\beta}{\alpha}.
\]
More generally, we are interested for which $\gamma \in  \RR^r_>$ the equations
\[
x^M = \gamma
\]
have a positive solution.
Taking the logarithm, we obtain the equivalent linear equations
\begin{equation}
\label{eq:lineq}
M^T \ln x = \ln \gamma ,
\end{equation}
which reduces the problem to linear algebra. 

In the rest of this section, we fix a matrix $M\in \RR^{n \times r}$ and write
\begin{equation*}
Z_{M,\gamma} = \{ x \in \RR^n_> \mid x^M = \gamma \}
\end{equation*}
for the set of all positive solutions with right-hand side $\gamma \in   \RR^r_>$.

\begin{proposition} \label{prop:exist}
The following statements hold: 
\[
 Z_{M,\gamma}\neq \emptyset \quad \text{for all } \gamma \in \RR^r_> \quad \text{iff} \quad \ker(M) = \{0\}. 
\]
 If $\ker(M) \neq \{0\}$, then 
\[
Z_{M,\gamma}\neq \emptyset \quad\text{for } \gamma \in \RR^r_> \quad \text{iff} \quad \gamma^C = 1,
\]
where $C \in \RR^{r \times p}$ with $\im(C) = \ker(M)$ and $\ker(C) = \{0\}$.
\end{proposition}
\begin{proof}
Using \eqref{eq:lineq}, $x^M = \gamma$ is equivalent to 
\[
\ln \gamma \in \im(M^T) = \ker(M)^\perp. 
\]
Hence, $Z_{M,\gamma}\neq \emptyset$ for all $\gamma \in \RR^r_>$ iff  $\ker(M)=\{0\}$.
If $\ker(M) \neq \{0\}$, then
\[
\ln \gamma \in \ker(M)^\perp = \im(C)^\perp
\quad \Leftrightarrow \quad
C^T \ln \gamma = 0
\quad \Leftrightarrow \quad
\gamma^C = 1. 
\]
\qed
\end{proof}

Computing an explicit positive solution $x^* \in  Z_{M,\gamma}$ (if it exists) in terms of $\gamma$ 
is equivalent to computing a particular solution for the linear equations \eqref{eq:lineq}.
For this, we use an arbitrary generalized inverse $H$ of $M^T$, that is, a matrix $H\in \RR^{n \times r}$ such that 
\[
M^T H M^T=M^T.
\]
We refer to \cite{Ben-IsraelGreville2003} for details on generalized inverses.

\begin{proposition} \label{prop:part}
Let $\gamma \in  \RR^r_>$ such that $\ln \gamma \in \im(M^T)$.  Let $H\in \RR^{n \times r}$ be a generalized inverse of $M^T$. Then,
\[
x^*=\gamma^{H^T}\in Z_{M,\gamma}.
\] 
\end{proposition}
\begin{proof}
By assumption, $\ln \gamma = M^T z$ for some $z \in \RR^n$.
Then,
\[
M^T \ln x^* = M^T H \ln \gamma = M^T H M^T z = M^T z = \ln \gamma 
\]
and hence $x^* \in Z_{M,\gamma}$ as claimed.
\qed
\end{proof}

Given one positive solution $x^* \in Z_{M,\gamma} $,
we have a generalized monomial parametrization for the set of all positive solutions.
\begin{proposition} \label{prop:para}
Let $x^* \in Z_{M,\gamma}$. Then,
\[
Z_{M,\gamma} = \{ x^* \circ \e^v \mid v \in \im(M)^\perp \} .
\]
If $\im(M)^\perp \neq \{0\}$, then
\[
Z_{M,\gamma} = \{ x^* \circ \xi^{B^T} \mid \xi \in \RR^q_> \},
\]
where $B \in \RR^{n \times q}$ with $\im(B) = \im(M)^\perp$ and $\ker(B) = \{0\}$.
\end{proposition}
\begin{proof}
The first equality follows from \eqref{eq:lineq}:
$x \in Z_{M,\gamma}$ iff  $v = \ln x - \ln x^* \in \ker(M^T)= \im(M)^\perp$, 
that is, $x = x^* \mal \e^v$ with $v \in  \im(M)^\perp$. 

Since the columns of $B$ form a basis for $\im(M)^\perp$,
we can write $v \in \im(M)^\perp$ uniquely as $v = B \, t$ for some $t \in \RR^q$.
By introducing $\xi = \e^t \in \RR^q_>$, we obtain
\[
(\e^v)_i = \e^{v_i} = \e^{\sum_j b_{ij} t_j} = \textstyle{\prod_j \xi_j^{b_{ij}}} = \xi^{b_{i}} = (\xi^{B^T})_i ,
\]
that is, $\e^v = \xi^{B^T}$. \qed
\end{proof}

Note that the conditions for the existence of positive solutions and the parametrization of all positive solutions, respectively,
depend only on the vector subspaces $\ker(M)$ and $\im(M)^\perp=\ker(M^T)$.

Summing up, we have seen that computing positive solutions for binomial equations reduces to linear algebra
involving the exponent matrix $M$.
The matrices $C$, $H$, and $B$ from Propositions~\ref{prop:exist}, \ref{prop:part}, and \ref{prop:para}
can be computed effectively if $M\in \QQ^{n \times r}$ and $C$, $B$ can be chosen to have only integer entries. 

Moreover, the linear algebra approach to binomial equations allows to deal algorithmically with indeterminate (symbolic) exponents.
We can use computer algebra methods for matrices with symbolic entries like Turing factoring (generalized PLU decomposition) \cite{CorlessJeffrey1997} and its implementation~\cite{CorlessJeffrey2013}.
Based on these methods,
we can compute explicit monomial parametrizations with symbolic exponents for generic entries and investigate conditions for special cases.
See Section \ref{sec:comp} for an example.  

\section{Kinetic Deficiency}
\label{sec:kin}

Applying the results from the previous section,
we rewrite the binomial equations~\eqref{eq:bineq} from our example,
\begin{equation*}
 K_1 \, x_{{3}}-K_2 \, (x_1)^\frac{1}{2}(x_2)^\frac{3}{2} = 0, \quad K_2 \, (x_1)^{3}-K_3 \, x_3  = 0, \quad K_4 \, x_4-K_5 \, x_1  = 0,
\end{equation*}
as 
\[
x^M = \kk ,
\]
where
\begin{equation}
\label{eq:M} 
M=
\begin{pmatrix}
-\frac{1}{2} &  3 & -1 \\
-\frac{3}{2} &  0 &  0 \\
           1 & -1 &  0 \\
           0 &  0 &  1
\end{pmatrix}
\end{equation}
and
\[
\kk = (K_2/K_1,K_3/K_2,K_5/K_4)^T ,
\]
which depends on the weights $k$ via the tree constants $K$.

Recall that the binomial equations depend on the basis vectors for $\ker(\ML)^\perp$
which are determined by the relation $\mathcal{E} = \{ (1,2),(2,3),(4,5) \}$.
To specify the resulting exponent matrix $M$ and the right-hand side $\kk$,
we have fixed an order on the relation.
By abuse of notation, we write
\[
\mathcal{E} = ( (1,2),(2,3),(4,5) ) .
\]
Hence, $M = \Yk \, I_\mathcal{E}$ with
\begin{equation} 
I_\mathcal{E}=
\begin{pmatrix}
-1 &  0 &  0 \\ 
 1 & -1 &  0 \\
 0 &  1 &  0 \\
 0 &  0 & -1 \\
 0 &  0 &  1 
\end{pmatrix}.
\end{equation}

In general, for a weakly reversible digraph with $m$ vertices and $l$ connected components,
let $\mathcal{E}$ be a relation as in Proposition \ref{prop:binomeq} with fixed order.
We denote by $I_\mathcal{E} \in \RR^{m \times (m-l)}$ the matrix with columns
\[
e^j-e^i \quad\text{for } (i,j) \in \mathcal{E},
\]
where $e^i$ denotes the $i$th standard basis vector in $\RR^m$.
Clearly, the columns of $I_\mathcal{E}$ are linearly independent and hence $\dim \im(I_\mathcal{E}) = m-l$.
To rewrite the binomial equations in Proposition~$\ref{prop:binomeq}$,
we define the exponent matrix $M \in \RR^{n \times (m-l)}$ as
\[
M=\Yk \, I_\mathcal{E} ,
\]
the right-hand side $\kk \in \RR^{m-l}_>$ as  
\begin{equation} \label{eq:kappa}
(\kk)_{(i,j)} = K_j/K_i  \quad\text{for } (i,j) \in \mathcal{E} ,
\end{equation}
and obtain
\[
Z_k = \{x \in \RR^n_> \mid x^M = \kk \} .
\]
We note that the actual matrix $M$ depends on $\mathcal{E}$, but $\im(M)$ does not.
This can be seen using the following fact.

\begin{proposition}
Let $G=(V,E)$ be a digraph with $m$ vertices and $l$ connected components.
Let $\Mi\in \RR^{m \times |E|}$ denote its incidence matrix (for fixed order on $E$),
and let $I_\mathcal{E} \in \RR^{m \times (m-l)}$ be as defined above.
Then, 
\[
\im(I_\mathcal{E}) = \im(\Mi) .
\]
\end{proposition}
\begin{proof}
From graph theory (see for example~\cite{Jungnickel2013}) and the argument above,
we know that $\dim \im(\Mi) = \dim \im(I_\mathcal{E}) = m-l$.
It remains to show that $\im(\Mi) \subseteq \im(I_\mathcal{E})$.
We consider the column $e^j-e^i$ of $\Mi$ corresponding to the edge $(i,j) \in E$.
Clearly, $i$ and $j$ are in the same connected component $L_\lambda$,
in particular, $i=i^\lambda_{\mu(i)}$ and $j=i^\lambda_{\mu(j)}$, where we assume $\mu(i) < \mu(j)$.
Then,
\[
e^j-e^i = \sum_{\mu=\mu(i),\ldots,\mu(j)-1} e^{i^\lambda_{\mu+1}} - e^{i^\lambda_{\mu}} ,
\]
where $e^{i^\lambda_{\mu+1}} - e^{i^\lambda_{\mu}}$ are columns of $I_\mathcal{E}$
corresponding to pairs $(i^\lambda_{\mu},i^\lambda_{\mu+1})$ in $\mathcal{E}$.
\qed
\end{proof}

Now, we see that $\im(M)$ equals the kinetic-order subspace $\tilde{S}$:
\[
\im(M) = \im (\Yk  I_\mathcal{E}) = \im(\Yk \Mi) = \tilde{S} .
\]

Finally,
we recall that the number of independent conditions on $\kk$
for the existence of a positive solution of $x^M=\kk$ is given by $\dim \ker(M)$, cf.~Proposition~\ref{prop:exist}.
Observing $M \in \RR^{n \times (m-l)}$, we obtain
\begin{equation}
\label{eq:kerM}
\dim \ker(M)= m-l - \dim \im(M) = m-l - \dim \tilde{S} .
\end{equation}

Hence, for a digraph with $m$ vertices and $l$ connected components, 
we define the {\em kinetic deficiency} as
\[
\tilde{\delta} = m-l - \tilde{s},
\]
where $\tilde{s} = \dim \tilde{S}$ denotes the dimension of the kinetic-order subspace.

\section{Computing Complex Balancing Equilibria}
\label{sec:comp}

Combining the results from the previous sections,
we obtain the following constructive characterization of complex balancing equilibria
in terms of quotients of tree constants. 

\begin{theorem} \label{thm}

Let $\ML$ be the graph Laplacian of a weakly reversible digraph with positive weights,
$m$ vertices, and $l$ connected components.
Let $\Yk \in \RR^{n \times m}$ be the matrix of kinetic complexes, $\tilde{s} = \dim \tilde{S}$ the dimension of the kinetic-order subspace,
and $\tilde{\delta} = m-l - \tilde{s}$ the kinetic deficiency.
Further, let $M\in \RR^{n \times (m-l)}$ and $\kk \in \RR_>^{m-l}$ such that 
\[
Z_k = \{ x \in \RR^n_> \mid \ML \, x^{\Yk}=0 \} = \{x \in \RR^n_> \mid x^M = \kk \}.
\]
Then, the following statements hold:

\begin{enumerate}
\item[{\rm(a)}] $Z_k\neq \emptyset$ for all $k$ iff $\tilde{\delta}=0$.
\smallskip
\item[{\rm(b)}] If $\tilde{\delta}>0$, then 
\[
Z_k\neq \emptyset\quad \text{iff}\quad (\kk)^C = 1, 
\]
where $C \in \RR^{(m-l) \times \tilde \delta}$ with $\im(C) = \ker(M)$ and $\ker(C) = \{0\}$.
\smallskip
\item[{\rm(c)}] If $Z_k \neq \emptyset$, then
\[
x^*=(\kk)^{H^T}\in Z_k,
\] 
where $H\in \RR^{n \times (m-l) }$ is a generalized inverse of $M^T$. 
\smallskip
\item[{\rm(d)}] If $x^*\in Z_k$ and $\tilde{s}<n$, then
\[
Z_k = \{ x^* \circ \xi^{B^T} \mid \xi \in \RR^{n-\tilde s}_> \},
\]
where $B \in \RR^{n \times (n-\tilde{s})}$ with $\im(B) =\tilde{S}^\perp$ and $\ker(B) = \{0\}$.
\end{enumerate}
\end{theorem}
\begin{proof}
By Propositions~\ref{prop:exist}, \ref{prop:part}, and \ref{prop:para}.
In fact, it remains to prove one implication in~(a).
Assume $Z_k\neq \emptyset$ for all $k$,
that is, there exists a solution to $x^M = \kk$ for all $k$.
By Lemma~\ref{lem} below, for all $\gamma \in \RR^{m-l}_>$, there exists $k$ such that $\kk = \gamma$.
Hence, there exists a solution to $x^M = \gamma$ for all $\gamma$.
Using \eqref{eq:kerM} and Proposition~\ref{prop:exist}, we obtain $\tilde{\delta} = \dim \ker(M) = 0$.
\qed
\end{proof}

\begin{lemma} \label{lem}
Let $\ML$ be the graph Laplacian of a weakly reversible digraph with positive weights,
$m$ vertices, and $l$ connected components,
and let $\kk \in \RR^{m-l}_>$ be the vector of quotients of tree constants defined in~\eqref{eq:kappa}.
For all $\gamma \in \RR^{m-l}_>$, there exists $k$ such that $\kk = \gamma$.
\end{lemma}
\begin{proof}
First, we show that every positive vector $\psi \in \RR^m_>$ solves $\ML \, \psi = 0$ for some weights $k$.
Indeed, for given $k$, the vector of tree constants $K \in \RR^m_>$ solves $\ML \, K = 0$,
and by choosing $k^*_{ij} = k_{ij} \, \frac{K_i}{\psi_i}$, 
one obtains
\begin{align*}
(A_{k^*} \, \psi)_i
&= \sum_{j=1}^m (A_{k^*})_{ij} \, \psi_j = \sum_{j \to i \in E} k^*_{ji} \, \psi_j - \sum_{i \to j \in E} k^*_{ij} \, \psi_i \\
&= \sum_{j \to i \in E} k_{ji} \, K_j - \sum_{i \to j \in E} k_{ij} \, K_i = \sum_{j=1}^m (\ML)_{ij} \, K_j = (\ML \, K)_i = 0
\end{align*}
for all $i=1,\ldots,m$, that is, $A_{k^*} \, \psi = 0$.

Let $\mathcal{E}$ be a relation as in Proposition \ref{prop:binomeq} with the obvious order.
Using basis vectors of $\ker(\ML)$ having tree constants as entries,
we find that
\[
\frac{\psi_j}{\psi_i} = \frac{K_j}{K_i} = (\kk)_{(i,j)}
\quad \text{for all } (i,j) \in \mathcal{E} .
\]
By choosing the entries of $\psi \in \RR^m_>$ in the obvious order,
every $\gamma \in \RR^{m-l}_>$ can be attained by $\kk$ for some $k$.
\qed
\end{proof}

\begin{remark}
Theorem~\ref{thm} is constructive in the following sense: 
\begin{itemize}
\item To test if the digraph $G$ is weakly reversible,
we compute the connected and the strongly connected components and check whether they are equal. 
\item The tree constants are computed in terms of the weights $k$,
using (fraction-free) Gaussian elimination on the sub-matrices of $A_k$ determined by the (strongly) connected components.
\item Given the kinetic complexes $\Yk \in \QQ^{n \times m}$ and the (strongly) connected components of the digraph,
we compute 
a matrix $M$ and a vector $\kk$ as introduced in Section~\ref{sec:kin}.
\item All matrices involved are computed by linear algebra from the exponent matrix $M$.
This can also be done algorithmically if the kinetic complexes $\Yk$ and hence $M$ contain indeterminate (symbolic) entries; 
see the end of Section~\ref{sec:possol}.
\end{itemize}
\end{remark}

In our example, $\tilde \delta = 5-2-3=0$ and a monomial parametrization of all complex balancing equilibria is given by
\[
\left( (\ka_3)^{-1}, (\ka_1)^{-\frac{2}{3}} \, (\ka_2)^{-\frac{2}{3}} \, (\ka_3)^{-\frac{5}{3}}, \ka_2^{-1} \, (\ka_3)^{-3}, 1 \right)^T
\circ
(\xi^3, \xi^5, \xi^9, \xi^3)^T ,
\]
where 
\[
\ka \equiv \kk = \left( \frac {k_{{12}}}{k_{{21}}+k_{{23}}}, \frac {k_{{23}}}{k_{{31}}},\frac {k_{{45}}}{k_{{54}}} \right)^T
\]
and $\xi \in \RR_>$.

To conclude,
we associate with each vertex of the graph
a kinetic complex possibly containing symbolic coefficients,
thereby specifying monomials with symbolic exponents:
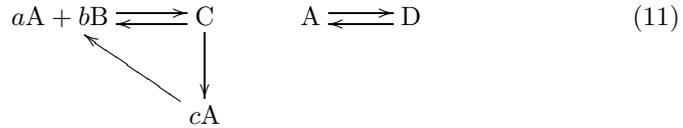
\begin{equation} \label{eq:symb}
\xymatrix{
a \ce{A} + b \ce{B} \ar@<0.5ex>[r] &  \ce{C} \ar@<0.5ex>[l] \ar[d] & \ce{A} \ar@<0.5ex>[r] & \ce{D} \ar@<0.5ex>[l] \\
& c \ce{A} \ar[lu] 
}
\end{equation}

In this setting, a monomial parametrization with symbolic exponents of all complex balancing equilibria is given by
\[
\left( (\ka_3)^{-1}, (\ka_1)^{-\frac{1}{b}} \, (\ka_2)^{-\frac{1}{b}} \, (\ka_3)^\frac{a-c}{b}, (\ka_2)^{-1} \, (\ka_3)^{-c}, 1 \right)^T
\circ
(\xi^{b}, {\xi}^{c-a}, {\xi}^{bc}, {\xi}^{b} )^T ,
\]
which is valid for non-zero $a,b,c\in \RR$.

\section{Generalized Birch's Theorem}

Since the dynamics of generalized mass-action systems is confined to cosets of the stoichiometric subspace,
we are interested in uniqueness and existence of complex balancing equilibria
in every positive stoichiometric compatibility class.

Let $G_k$ be a weakly reversible digraph with positive weights, $m$ vertices and $l$ connected components.
For fixed rate constants $k$,
a complex balancing equilibrium $x^* \in \RR^n_>$ of the mass-action system $(G_k,\ys,\yk)$ solves $\ML \, x^{\Yk} = 0$,
where $\ML \in \RR^{m \times m}$ is the graph Laplacian
and $\Yk \in \RR^{n \times m}$ is the matrix of kinetic complexes.
Equivalently,
it solves $x^M = \kk$,
where the columns of $M \in \RR^{n \times (m-l)}$ are differences of kinetic complexes
and the entries of $\kk \in \RR^{m-l}_>$ are quotients of the tree constants $K$, which depend on the weights $k$.
In other words,
\begin{align*}
Z_k &= \{ x \in \RR^n_> \mid \ML \, x^{\Yk} = 0 \} \\
&= \{ x \in \RR^n_> \mid x^M = \kk \} .
\end{align*}
Given a complex balancing equilibrium $x^* \in \RR^n_>$,
we further know that
\begin{align*}
Z_k &= \{ x^* \circ \e^v \mid v \in \im(M)^\perp \} \\
&= \{ x^* \circ \xi^{B^T} \mid \xi \in \RR^{\tilde{d}}_> \} ,
\end{align*}
where the second equality holds if $\im(M)^\perp \neq \{0\}$
and $B \in \RR^{n \times \tilde{d}}$ is defined as $\im(B) = \im(M)^\perp$ and $\ker(B) = \{0\}$.

For simplicity,
we write $\tilde{\mat} = B^T \in \RR^{\tilde{d} \times n}$
such that $\tilde{S} = \im(M) = \im(B)^\perp = \im(\tilde{\mat}^T)^\perp = \ker(\tilde{\mat})$.
Analogously, 
we introduce a matrix $\mat \in \RR^{d \times n}$ with full rank $d$ such that $S = \ker(\mat)$.

If the intersection of the set of complex balancing equilibria with some compatibility class,
\[
Z_k \cap (x'+S) ,
\]
is non-empty,
then there exist $\xi \in \RR^{\tilde{d}}_>$ and $u \in S$ such that
\[
x^* \mal \xi^{\tilde{\mat}} = x' + u .
\]
Multiplication by $\mat$ yields
\[
\mat \, (x^* \mal \xi^{\tilde{\mat}}) = \mat \, x' 
\]
such that
existence and uniqueness of complex balancing equilibria in every stoichiometric compatibility class
are equivalent to surjectivity and injectivity of the generalized polynomial map 
\begin{align} \label{eq:gpm}
f_{x^*} \colon & \RR^{\tilde{d}}_> \to C^\circ \subseteq \RR^d \\
& \xi \mapsto \mat \, (x^* \mal \xi^{\tilde{\mat}})
= \sum_{i=1}^n x^*_i \; \xi^{\tilde{\vct}^i} \vct^i , \nonumber
\end{align}
where $C^\circ$ is the interior of the polyhedral cone
\begin{equation*}
C = \left\{ \mat x'  \in \RR^d \mid x' \in \RR^n_\ge \right\} = \left\{ \sum_{i=1}^n x'_i \, \vct^i \in \RR^d \mid x' \in \RR^n_\ge \right\} .
\end{equation*}

In mass-action systems,
where $S=\tilde{S}$ and hence $\mat=\tilde{\mat}$,
one version~\cite{Fulton1993} of Birch's theorem~\cite{Birch1963} states
that $f_{x^*}$ is a real analytic isomorphism of $\RR^d_>$ onto $C^\circ$ for all $x^* \in \RR^n_>$.
We refer to~\cite[Sect.~5]{GopalkrishnanMillerShiu2013} for a recent overview on the use of Birch's theorem in CRNT
and to \cite{PachterSturmfels2005} for the version used in algebraic statistics.
Interestingly, Martin W.~Birch's seminal paper on maximum likelihood methods for log-linear models
was part of a PhD thesis at the University of Glasgow that was never submitted~\cite{Fienberg1992}.

Recently, we have generalized Birch's theorem to $\mat \neq \tilde{\mat}$, cf.~\cite[Proposition~3.9]{MuellerRegensburger2012}.
To formulate the result,
we define the sign vector $\sigma(x) \in \{ -,0,+ \}^n$ of a vector $x \in \RR^n$
by applying the sign function component-wise,
and we write $\sigma(S) = \{ \sigma(x) \mid x \in S \}$ for a subset $S \subseteq \RR^n$.

\begin{theorem}
\label{thm:Birch}
Let $\mat \in \RR^{d \times n}$, $\tilde{\mat} \in \RR^{\tilde{d} \times n}$
and $S = \ker(\mat)$, $\tilde{S} = \ker(\tilde{\mat})$.
If $\sigma(S)=\sigma(\tilde{S})$ and $(+,\ldots,+)^T \in \sigma(S^\perp)$,
then the generalized polynomial map $f_{x^*}$ in \eqref{eq:gpm}
is a real analytic isomorphism of $\RR^{\tilde{d}}_>$ onto $C^\circ$ for all $x^* \in \RR^n_>$.
\end{theorem}

If $\tilde{\delta} = 0$, there exists a complex balancing equilibrium for all rate constants $k$, by Theorem~\ref{thm}.
If further the generalized polynomial map $f_{x^*}$ is surjective and injective for all $x^*$,
then, by Theorem~\ref{thm:Birch}, there exists a unique steady state in every positive stoichiometric compatibility class for all $k$.

\bigskip

To illustrate the result, we consider the minimal (weakly) reversible weighted digraph
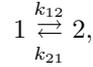
\[
1 \underset{k_{21}}{\overset{k_{12}}{\rightleftarrows}} 2 ,
\]
and associate with each vertex a (stoichiometric) complex
\[
\ce{A} + \ce{B} \rightleftarrows \ce{C} 
\]
as well as a kinetic complex
\[
a \ce{A} + b \ce{B} \rightleftarrows \ce{C} ,
\]
where $a,b > 0$.
We find $S=\im (-1, -1, 1)^T$ and $\tilde{S} = \im (-a, -b , 1)^T$ and choose
\[
\mat=\begin{pmatrix} 1 &  0 & 1 \\ 0 & 1 & 1\end{pmatrix}
\quad \text{and} \quad 
\tilde{\mat}=\begin{pmatrix} 1 & 0 & a \\ 0 & 1 & b \end{pmatrix}
\]
such that $S = \ker(\mat)$ and $\tilde{S} = \ker(\tilde{\mat})$.
Clearly, our generalization of Birch's theorem applies since
\[
\sigma(S) = \left\{
\begin{pmatrix}-\\-\\+\end{pmatrix} ,
\begin{pmatrix}+\\+\\-\end{pmatrix} ,
\begin{pmatrix}0\\0\\0\end{pmatrix}
\right\}
=\sigma(\tilde{S})
\]
and $(1,1,2)^T \in S^\perp$.
Hence,
there exists a unique solution $\xi \in \RR^2_>$ for the system of generalized polynomial equations
\[
x^*_1\;\xi_1
\begin{pmatrix}
1 \\ 0
\end{pmatrix}
+
x^*_2\;\xi_2
\begin{pmatrix}
0 \\ 1
\end{pmatrix}
+
x^*_3\;(\xi_1)^a\,(\xi_2)^b
\begin{pmatrix}
1 \\ 1
\end{pmatrix}
=
\begin{pmatrix}
y_1 \\ y_2
\end{pmatrix}
\]
for all right-hand-sides $y \in C^\circ = \RR^2_>$, all parameters $x^* \in \RR^3_>$,
and all exponents $a,b > 0$. 
Note that Birch's theorem guarantees the existence of a unique solution only for $a=b=1$.

In terms of the generalized mass-action system above, we have the following result:
Since $\tilde{\delta}=2-1-1=0$,
there exists a unique complex balancing equilibrium in every positive stoichiometric compatibility class for all $k_{12},k_{21}>0$
and all kinetic orders $a,b > 0$.
Since $\delta=2-1-1=0$, there are no other steady states.

\section{Sign Vectors and Oriented Matroids}

The characterization of surjectivity and injectivity of generalized polynomial maps involves 
sign vectors of real linear subspaces, which are basic examples of oriented matroids.
(Whereas a matroid abstracts the notion of linear independence,
an oriented matroid additionally captures orientation.) 

The theory of oriented matroids provides a common framework to study combinatorial properties of various geometric objects,
including point configurations, hyperplane arrangements, convex polyhedra, and directed graphs. 
See~\cite{BachemKern1992}, \cite[Chapters 6 and 7]{Ziegler1995}, and~\cite{Richter-GebertZiegler1997}
for an introduction and overview, and \cite{BjornerLasSturmfelsWhiteZiegler1999} for a comprehensive study.

There are several sets of sign vectors associated with a linear subspace
which satisfy the axiom systems for (co-)vectors, (co-)circuits, or chirotopes of oriented matroids.
(In fact, there are non-realizable oriented matroids that do not arise from linear subspaces.)

For algorithmic purposes, the characterization of oriented matroids in terms of basis orientations is most useful.
The chirotope of a matrix $\mat \in \RR^{d \times n}$ (with rank $d$) is defined as the map
\begin{align*}
\chi_\mat \colon \{1, \ldots , n\}^d &\to \{-, 0, +\} \\
(i_1 , \ldots , i_d ) &\mapsto \sign(\det(\vct^{i_1} , \ldots , \vct^{i_d} )) ,
\end{align*}
which records for each $d$-tuple of vectors whether it forms a positively oriented basis of $\RR^d$,
a negatively oriented basis, or not a basis. 
Hence, chirotopes can be used to test algorithmically
if the sign vectors of two subspaces are equal by comparing determinants of maximal minors.

More generally,
the realization space of matrices defining the same oriented matroid as $\mat \in \RR^{d \times n}$ (with rank $d$)
is described by the semi-algebraic set   
\begin{align*}
\mathcal{R}(\mat) = 
\{ A \in \RR^{d \times n} \mid 
& \sign(\det(a^{i_1}, \ldots , a^{i_d} )) = \\
& \sign(\det(\vct^{i_1} , \ldots , \vct^{i_d} )), \; 1\leq i_1 < \cdots < i_d \leq n \} .
\end{align*}
Mn\"{e}v's universality theorem~\cite{Mnev1988} theorem states that already for oriented matroids with rank $d=3$,
the realization space can be ``arbitrarily complicated'';
see \cite{BjornerLasSturmfelsWhiteZiegler1999} for a precise statement and \cite{Basu2006} for semi-algebraic sets and algorithms.

Concerning software, the \texttt{C++} package \href{http://www.rambau.wm.uni-bayreuth.de/TOPCOM}{TOPCOM} \cite{Rambau2002}
allows to compute efficiently chirotopes with rational arithmetic and generate all cocircuits (covectors with minimal support).
There is also an interface to the open source computer algebra system \href{http://www.sagemath.org/}{SAGE}. 

In our running example,
we have $\tilde{S} = \im(\Yk \, I_\mathcal{E}) = \im(M)$ with $M$ as in~\eqref{eq:M}.
Analogously, $S = \im (\Ys \, I_\mathcal{E}) = \im(\mathcal{N})$ with
\begin{equation} \label{eq:N}
\mathcal{N} =
\begin{pmatrix}
-1 &  2 & -1 \\ 
-1 &  0 &  0 \\
 1 & -1 &  0 \\ 
 0 &  0 &  1
\end{pmatrix} .
\end{equation}
To check the sign vector condition $\sigma(S)=\sigma(\tilde{S})$,
we compare the chirotopes of $\mathcal{N}^T$ and $M^T$.
Computing the signs of the four maximal minors of $\mathcal{N}^T$, we see that its chirotope is given by
\[
\chi_{\mathcal{N}^T}(1,2,3)=-, \quad \chi_{\mathcal{N}^T}(1,2,4)=+, \quad \chi_{\mathcal{N}^T}(1,3,4)=-, \quad \chi_{\mathcal{N}^T}(2,3,4)=+ .
\]
Analogously, we compute the chirotope of $M^T$ and verify $\chi_{\mathcal{N}^T}=\chi_{M^T}$.
Clearly, the other sign vector condition $(+,\ldots,+)^T \in \sigma(S^\perp)$ also holds, for example, $(1,1,2,1)^T\in S^\perp$. 

Since $\tilde{\delta}=0$, we know from Theorems~\ref{thm} and \ref{thm:Birch}
that there exists a unique complex balancing equilibrium
in every positive stoichiometric compatibility class for all rate constants $k$.
Moreover, since $\delta=5-2-3=0$,
we know that there are no steady states other than complex balancing equilibria for the ODE \eqref{eq:ODE}. 

In the setting of symbolic exponents~\eqref{eq:symb}, the exponent matrix amounts to
\begin{equation} \label{eq:Msymb}
M =
\begin{pmatrix}
-a &  c & -1 \\ 
-b &  0 &  0 \\ 
 1 & -1 &  0 \\ 
 0 &  0 &  1
\end{pmatrix}
\end{equation}
and the chirotope of $M^T$ (in the same order as above) is given by
\[
-\sign(b), \quad \sign(b\,c), \quad \sign(a-c), \quad \sign(b)
\]
for $a,b,c \neq 0$.
Hence, there exists a unique steady state in every positive stoichiometric compatibility class for all rate constants and all exponents with $a,b,c>0$ and $a<c$.

\section{Multistationarity}

A (generalized) chemical reaction network $(G,y,\tilde{y})$ has \emph{the capacity for multistationarity}
if there exist rate constants $k$ such that the generalized mass action system $(G_k,y,\tilde{y})$
admits more than one steady state in some stoichiometric compatibility class.

In mass-action systems,
every stoichiometric compatibility class contains at most one complex balancing equilibrium.
However, in generalized mass action systems, multiple steady states of this type are possible~\cite[Proposition~3.2]{MuellerRegensburger2012}.

\begin{proposition} \label{prop:multi}
Let $(G,y,\tilde{y})$ be a generalized chemical reaction network.
If $G$ is weakly reversible and $\sigma(S) \cap \sigma(\tilde{S}^\perp) \neq \{0\}$,
then $(G,y,\tilde{y})$
has the capacity for multiple complex balancing equilibria.
\end{proposition}

Analogously,
multiple toric steady states are possible (for networks with mass-action kinetics)
if the sign vectors of two subspaces intersect non-trivially \cite{Conradi2008,PerezMillanDickensteinShiuConradi2012}.
For deficiency one networks (with mass-action kinetics),
the capacity for multistationarity
is also characterized by sign conditions \cite{Feinberg1988,Feinberg1995b}. 

For precluding multistationarity,
injectivity of the right-hand side of the dynamical system on cosets of the stoichiometric subspace is sufficient.
In~\cite{Muelleretal2013},
we characterize injectivity of generalized polynomial maps on cosets of the stoichiometric subspace in terms of sign vectors.
There, we also give a survey on injectivity criteria
and discuss algorithms to check sign vector conditions.

For the last time, we return to our example,
in particular, to the setting of symbolic kinetic complexes.
Considering the matrix $M$ in \eqref{eq:Msymb},
a matrix $B$ with $\im(B) = \im(M)^\perp = \tilde{S}^\perp$ is given by
\[
B=(b, c-a, b\,c, b)^T
\]
for $a,b,c \neq 0$.
Hence, for $a,b,c > 0$ and $a>c$, we have $(+,-,+,+)^T \in \sigma(\tilde{S}^\perp)$.

On the other hand, considering the matrix $\mathcal{N}$ in \eqref{eq:N} with $\im(\mathcal{N}) = S$, we also have $(+,-,+,+)^T \in \sigma(S)$,
and hence $\sigma(S) \cap \sigma(\tilde{S}^\perp) \neq \{0\}$.
By Proposition~\ref{prop:multi}, if the inequalities $a,b,c > 0$ and $a>c$ hold,
then there exist rate constants $k$ that admit more than one complex balancing equilibrium in some stoichiometric compatibility class.


\begin{thebibliography}{10}

\bibitem{AdrovicVerschelde2013}
Adrovic, D., Verschelde, J.:
\newblock A polyhedral method to compute all affine solution sets of sparse
  polynomial systems.
\newblock (2013) \href{http://http://arxiv.org/abs/1310.4128}{arXiv:1310.4128}
  [cs.SC].

\bibitem{BachemKern1992}
Bachem, A., Kern, W.:
\newblock Linear programming duality.
\newblock Springer-Verlag, Berlin (1992)

\bibitem{Basu2006}
Basu, S., Pollack, R., Roy, M.F.:
\newblock Algorithms in real algebraic geometry. Second edn.
\newblock Springer-Verlag, Berlin (2006)

\bibitem{Ben-IsraelGreville2003}
Ben-Israel, A., Greville, T.N.E.:
\newblock Generalized inverses. Second edn.
\newblock Springer-Verlag, New York (2003)

\bibitem{Birch1963}
Birch, M.W.:
\newblock Maximum likelihood in three-way contingency tables.
\newblock J. Roy. Statist. Soc. Ser. B \textbf{25} (1963)  220--233

\bibitem{BjornerLasSturmfelsWhiteZiegler1999}
Bj{\"o}rner, A., Las~Vergnas, M., Sturmfels, B., White, N., Ziegler, G.M.:
\newblock Oriented matroids. Second edn.
\newblock Cambridge University Press, Cambridge (1999)

\bibitem{BoulierLemairePetitotSedoglavic2011}
{Boulier}, F., {Lemaire}, F., {Petitot}, M., {Sedoglavic}, A.:
\newblock {Chemical reaction systems, computer algebra and systems biology.
  (Invited talk).}
\newblock In: {CASC 2011}. LNCS, vol. 6885.
\newblock Berlin: Springer (2011)  73--87

\bibitem{BrualdiRyser1991}
Brualdi, R.A., Ryser, H.J.:
\newblock Combinatorial matrix theory.
\newblock Cambridge University Press, Cambridge (1991)

\bibitem{Conradi2008}
Conradi, C., Flockerzi, D., Raisch, J.:
\newblock Multistationarity in the activation of a {MAPK}: parametrizing the
  relevant region in parameter space.
\newblock Math. Biosci. \textbf{211} (2008)  105--131

\bibitem{CorlessJeffrey1997}
Corless, R.M., Jeffrey, D.J.:
\newblock The turing factorization of a rectangular matrix.
\newblock SIGSAM Bull. \textbf{31} (1997)  20--30

\bibitem{CorlessJeffrey2013}
Corless, R.M., Jeffrey, D.J.:
\newblock Linear Algebra in Maple.
\newblock In: CRC Handbook of Linear Algebra. Second edition edn. Chapman and
  Hall/CRC (2013)

\bibitem{CraciunDickensteinShiuSturmfels2009}
Craciun, G., Dickenstein, A., Shiu, A., Sturmfels, B.:
\newblock Toric dynamical systems.
\newblock J. Symbolic Comput. \textbf{44} (2009)  1551--1565

\bibitem{Dickenstein2009}
Dickenstein, A.:
\newblock A world of binomials.
\newblock In: Foundations of computational mathematics, {H}ong {K}ong 2008.
\newblock Cambridge Univ. Press, Cambridge (2009)  42--67

\bibitem{ErramiSeilerEiswirthWeber2012}
{Errami}, H., {Seiler}, W.M., {Eiswirth}, M., {Weber}, A.:
\newblock {Computing Hopf bifurcations in chemical reaction networks using
  reaction coordinates.}
\newblock In: CASC 2012. LNCS, vol. 7442.
\newblock Berlin: Springer (2012)  84--97

\bibitem{Feinberg1972}
Feinberg, M.:
\newblock Complex balancing in general kinetic systems.
\newblock Arch. Rational Mech. Anal. \textbf{49} (1972/73)  187--194

\bibitem{Feinberg1979}
Feinberg, M.:
\newblock Lectures on chemical reaction networks.
\newblock Available online at
  \url{http://crnt.engineering.osu.edu/LecturesOnReactionNetworks} (1979)

\bibitem{Feinberg1987}
Feinberg, M.:
\newblock {Chemical reaction network structure and the stability of complex
  isothermal reactors--{I}. The deficiency zero and deficiency one theorems}.
\newblock Chem. Eng. Sci. \textbf{42} (1987)  2229--2268

\bibitem{Feinberg1988}
Feinberg, M.:
\newblock Chemical reaction network structure and the stability of complex
  isothermal reactors--{II}. multiple steady states for networks of deficiency
  one.
\newblock Chem. Eng. Sci. \textbf{43} (1988)  1--25

\bibitem{Feinberg1995a}
Feinberg, M.:
\newblock The existence and uniqueness of steady states for a class of chemical
  reaction networks.
\newblock Arch. Rational Mech. Anal. \textbf{132} (1995)  311--370

\bibitem{Feinberg1995b}
Feinberg, M.:
\newblock Multiple steady states for chemical reaction networks of deficiency
  one.
\newblock Arch. Rational Mech. Anal. \textbf{132} (1995)  371--406

\bibitem{FeinbergHorn1977}
Feinberg, M., Horn, F.J.M.:
\newblock Chemical mechanism structure and the coincidence of the
  stoichiometric and kinetic subspaces.
\newblock Arch. Rational Mech. Anal. \textbf{66} (1977)  83--97

\bibitem{Fienberg1992}
Fienberg, S.E.:
\newblock Introduction to {B}irch (1963) {M}aximum likelihood in three-way
  contingency tables.
\newblock In Kotz, S., Johnson, N.L., eds.: Breakthroughs in statistics. {V}ol.
  {II}.
\newblock Springer-Verlag, New York (1992)  453--461

\bibitem{Fulton1993}
Fulton, W.:
\newblock Introduction to toric varieties.
\newblock Princeton University Press, Princeton, NJ (1993)

\bibitem{GatermannWolfrum2005}
Gatermann, K., Wolfrum, M.:
\newblock Bernstein's second theorem and {V}iro's method for sparse polynomial
  systems in chemistry.
\newblock Adv. in Appl. Math. \textbf{34} (2005)  252--294

\bibitem{Gatermann2001}
Gatermann, K.:
\newblock Counting stable solutions of sparse polynomial systems in chemistry.
\newblock In: Symbolic computation: solving equations in algebra, geometry, and
  engineering.
\newblock Amer. Math. Soc., Providence, RI (2001)  53--69

\bibitem{GatermannEiswirthSensse2005}
Gatermann, K., Eiswirth, M., Sensse, A.:
\newblock Toric ideals and graph theory to analyze {H}opf bifurcations in mass
  action systems.
\newblock J. Symbolic Comput. \textbf{40} (2005)  1361--1382

\bibitem{GatermannHuber2002}
Gatermann, K., Huber, B.:
\newblock A family of sparse polynomial systems arising in chemical reaction
  systems.
\newblock J. Symbolic Comput. \textbf{33} (2002)  275--305

\bibitem{GopalkrishnanMillerShiu2013}
Gopalkrishnan, M., Miller, E., Shiu, A.:
\newblock A {G}eometric {A}pproach to the {G}lobal {A}ttractor {C}onjecture.
\newblock SIAM J. Appl. Dyn. Syst. \textbf{13} (2014)  758--797

\bibitem{GrigorievWeber2012}
{Grigoriev}, D., {Weber}, A.:
\newblock {Complexity of solving systems with few independent monomials and
  applications to mass-action kinetics.}
\newblock In: CASC 2012. LNCS, vol. 7442.
\newblock Berlin: Springer (2012)  143--154

\bibitem{Gunawardena}
Gunawardena, J.:
\newblock Chemical reaction network theory for in-silico biologists (2003)
  Available online at \url{http://vcp.med.harvard.edu/papers/crnt.pdf}.

\bibitem{Gunawardena2012}
Gunawardena, J.:
\newblock A linear framework for time-scale separation in nonlinear biochemical
  systems.
\newblock PLoS ONE \textbf{7} (2012)  e36321

\bibitem{Horn1972}
Horn, F.:
\newblock Necessary and sufficient conditions for complex balancing in chemical
  kinetics.
\newblock Arch. Rational Mech. Anal. \textbf{49} (1972/73)  172--186

\bibitem{HornJackson1972}
Horn, F., Jackson, R.:
\newblock General mass action kinetics.
\newblock Arch. Rational Mech. Anal. \textbf{47} (1972)  81--116

\bibitem{Johnston2013}
Johnston, M.D.:
\newblock Translated {C}hemical {R}eaction {N}etworks.
\newblock Bull. Math. Biol. \textbf{76}(5) (2014)  1081--1116

\bibitem{Jungnickel2013}
Jungnickel, D.:
\newblock Graphs, networks and algorithms. Fourth edn.
\newblock Springer, Heidelberg (2013)

\bibitem{Lemaire2012}
Lemaire, F., {\"U}rg{\"u}pl{\"u}, A.:
\newblock M{ABS}ys: modeling and analysis of biological systems.
\newblock In: Algebraic and numeric biology. LNCS, vol. 6479.
\newblock Springer, Heidelberg (2012)  57--75

\bibitem{MirzaevGunawardena2013}
Mirzaev, I., Gunawardena, J.:
\newblock Laplacian dynamics on general graphs.
\newblock Bull. Math. Biol. \textbf{75} (2013)  2118--2149

\bibitem{Mnev1988}
Mn{\"e}v, N.E.:
\newblock The universality theorems on the classification problem of
  configuration varieties and convex polytopes varieties.
\newblock In: Topology and geometry---{R}ohlin {S}eminar. Volume 1346 of Lecture Notes in Math.
\newblock Springer, Berlin (1988)  527--543

\bibitem{Muelleretal2013}
M\"uller, S., Feliu, E., Regensburger, G., Conradi, C., Shiu, A., Dickenstein,
  A.:
\newblock Sign conditions for injectivity of generalized polynomial maps with
  applications to chemical reaction networks and real algebraic geometry.
\newblock (2013) Submitted.
  \href{http://arxiv.org/abs/1311.5493}{arXiv:1311.5493} [math.AG].

\bibitem{MuellerRegensburger2012}
M\"uller, S., Regensburger, G.:
\newblock Generalized mass action systems: {C}omplex balancing equilibria and
  sign vectors of the stoichiometric and kinetic-order subspaces.
\newblock SIAM J. Appl. Math. \textbf{72} (2012)  1926--1947

\bibitem{PachterSturmfels2005}
Pachter, L., Sturmfels, B.:
\newblock Statistics.
\newblock In: Algebraic statistics for computational biology.
\newblock Cambridge Univ. Press, New York (2005)  3--42

\bibitem{PerezMillanDickensteinShiuConradi2012}
P{\'e}rez~Mill{\'a}n, M., Dickenstein, A., Shiu, A., Conradi, C.:
\newblock Chemical reaction systems with toric steady states.
\newblock Bull. Math. Biol. \textbf{74} (2012)  1027--1065

\bibitem{Rambau2002}
Rambau, J.:
\newblock T{OPCOM}: triangulations of point configurations and oriented
  matroids.
\newblock In: Mathematical software ({B}eijing, 2002).
\newblock World Sci. Publ., River Edge, NJ (2002)  330--340

\bibitem{Richter-GebertZiegler1997}
Richter-Gebert, J., Ziegler, G.M.:
\newblock Oriented matroids.
\newblock In: Handbook of discrete and computational geometry.
\newblock CRC, Boca Raton, FL (1997)  111--132

\bibitem{SamalErramiWeber2012}
{Samal}, S.S., {Errami}, H., {Weber}, A.:
\newblock {PoCaB: a software infrastructure to explore algebraic methods for
  bio-chemical reaction networks.}
\newblock In: CASC 2012. LNCS, vol. 7442.
\newblock Berlin: Springer (2012)  294--307

\bibitem{Savageau1969b}
Savageau, M.A.:
\newblock {B}iochemical systems analysis: {II}. {T}he steady state solutions
  for an n-pool system using a power-law approximation.
\newblock J. Theor. Biol. \textbf{25} (1969)  370--379

\bibitem{ThomsonGunawardena2009}
Thomson, M., Gunawardena, J.:
\newblock The rational parameterisation theorem for multisite
  post-translational modification systems.
\newblock J. Theoret. Biol. \textbf{261} (2009)  626--636

\bibitem{Voit2013}
Voit, E.O.:
\newblock {Biochemical systems theory: a review.}
\newblock {ISRN Biomath.} \textbf{2013} (2013) 53

\bibitem{Zeilberger1985}
Zeilberger, D.:
\newblock A combinatorial approach to matrix algebra.
\newblock Discrete Math. \textbf{56} (1985)  61--72

\bibitem{Ziegler1995}
Ziegler, G.M.:
\newblock Lectures on polytopes.
\newblock Springer-Verlag, New York (1995)

\end{thebibliography}

\end{document}